\documentclass[11pt,reqno,tbtags,a4paper]{amsart}
\usepackage{amssymb}
\usepackage{xpunctuate}
\usepackage{amsbsy}
\usepackage{url}
\usepackage{graphicx}
\usepackage[square,numbers]{natbib}
\bibpunct[, ]{[}{]}{;}{n}{,}{,}

\title%[]
{Short cycles in high genus unicellular maps}

\date{3 March, 2021}
\author{Svante Janson}
\thanks{Supported by the Knut and Alice Wallenberg Foundation}
\address{Department of Mathematics, Uppsala University, PO Box 480,
SE-751~06 Uppsala, Sweden}
\email{svante.janson@math.uu.se}
\newcommand\urladdrx[1]{{\urladdr{\def~{{\tiny$\sim$}}#1}}}
\urladdrx{http://www2.math.uu.se/~svante/}
\author{Baptiste Louf}
\email{baptiste.louf@math.uu.se}
\urladdrx{https://baptiste.louf.fr}

\subjclass[2010]{} 
\numberwithin{equation}{section}

\allowdisplaybreaks

\theoremstyle{plain}% default
\newtheorem{theorem}{Theorem}[section]
\newtheorem{lemma}[theorem]{Lemma}
\newtheorem{proposition}[theorem]{Proposition}
\newtheorem{corollary}[theorem]{Corollary}

\theoremstyle{definition}

\newtheorem{exampleqqq}[theorem]{Example}

\newtheorem{remarkqqq}[theorem]{Remark}
\newenvironment{remark}{\begin{remarkqqq}}
  {\hfill\qedsymbol\end{remarkqqq}}
\theoremstyle{remark}

\newenvironment{romenumerate}[1][-10pt]{% optional argument changes indentation
\addtolength{\leftmargini}{#1}\begin{enumerate}% gives (i), (ii) etc.
 \renewcommand{\labelenumi}{\textup{(\roman{enumi})}}%
 }{\end{enumerate}}

\newenvironment{PXenumerate}[1]{%  argument yields prefix
\begin{enumerate}% gives (#1 1), (#1 2) etc.
 \renewcommand{\labelenumi}{\textup{(#1\arabic{enumi})}}%
 }{\end{enumerate}}

\newcounter{oldenumi}
{\setcounter{oldenumi}{\value{enumi}}
\begin{romenumerate} \setcounter{enumi}{\value{oldenumi}}}
{\end{romenumerate}}

\newcounter{thmenumerate}

\newcounter{xenumerate}   %no left indentation; thus wider lines

\newcommand{\refT}[1]{Theorem~\ref{#1}}

\newcommand{\refL}[1]{Lemma~\ref{#1}}
\newcommand{\refLs}[1]{Lemmas~\ref{#1}}

\newcommand{\refS}[1]{Section~\ref{#1}}

\newcommand\marginal[1]{\marginpar[\raggedleft\tiny #1]{\raggedright\tiny#1}}
\newcommand\REM[1]{{\raggedright\texttt{[#1]}\par\marginal{XXX}}}
\newcommand\XREM[1]{\relax}

\begingroup
  \divide\count255 by 60
  \multiply\count255 by -60
  \advance\count255 by \time
  \ifnum \count255 < 10 \xdef\klockan{\the\count1.0\the\count255}
  \else\xdef\klockan{\the\count1.\the\count255}\fi
\endgroup

\DeclareMathOperator*{\sumx}{\sum\nolimits^{*}}

\DeclareMathOperator*{\hsumx}{\widehat{\sum}}
\newcommand{\sumio}{\sum_{i=0}^\infty}

\newcommand{\sumlo}{\sum_{\ell=0}^\infty}

\newcommand{\sumk}{\sum_{k=1}^\infty}
\newcommand{\suml}{\sum_{\ell=1}^\infty}

\newcommand{\sumim}{\sum_{i=1}^m}

\newcommand{\sumlk}{\sum_{\ell=1}^k}

\newcommand{\prodik}{\prod_{i=1}^k}

\newcommand\set[1]{\ensuremath{\{#1\}}}

\newcommand\xpar[1]{(#1)}
\newcommand\bigpar[1]{\bigl(#1\bigr)}
\newcommand\Bigpar[1]{\Bigl(#1\Bigr)}

\newcommand\bigsqpar[1]{\bigl[#1\bigr]}

\newcommand\xcpar[1]{\{#1\}}

\newcommand\Bigabs[1]{\Bigl\lvert#1\Bigr\rvert}

\def\rompar(#1){\textup(#1\textup)}    % usage: \rompar(...)

\newcommand\Bigparfrac[2]{\Bigpar{\frac{#1}{#2}}}

\def\xexp(#1){e^{#1}}

\newcommand\ntoo{\ensuremath{{n\to\infty}}}

\newcommand\mtoo{\ensuremath{{m\to\infty}}}

\newcommand\upto{\nearrow}

\newcommand\start{\text{start}}

\newcommand\eend{\text{end}}

\newcommand\punkt{\xperiod}    % xpunctuate
\newcommand\iid{i.i.d\punkt}    
\newcommand\ie{i.e\punkt}
\newcommand\eg{e.g\punkt}

\newcommand{\as}{a.s\punkt}

\newcommand{\tend}{\longrightarrow}
\newcommand\dto{\overset{\mathrm{d}}{\tend}}
\newcommand\pto{\overset{\mathrm{p}}{\tend}}

\newcommand\bbR{\mathbb R}

\newcommand\bbZ{\mathbb Z}

\newcommand\E{\operatorname{\mathbb E{}}}
\renewcommand\P{\operatorname{\mathbb P{}}}
\newcommand\Var{\operatorname{Var}}
\newcommand\Cov{\operatorname{Cov}}

\newcommand\Po{\operatorname{Po}}

\newcommand\Ge{\operatorname{Ge}}

\newcommand\ga{\alpha}
\newcommand\gb{\beta}

\newcommand\gD{\Delta}
\newcommand\gf{\varphi}
\newcommand\gam{\gamma}

\newcommand\kk{\chi}
\newcommand\gl{\lambda}

\newcommand\go{\omega}

\newcommand\gs{\sigma}

\newcommand\gth{\theta}
\newcommand\eps{\varepsilon}

\renewcommand\phi{\xxx}  %% WARNING

\newcommand\cC{\mathcal C}

\newcommand\cE{\mathcal E}

\newcommand\cP{\mathcal P}

\newcommand\cT{{\mathcal T}}

\newcommand\indic[1]{\boldsymbol1\xcpar{#1}}

\newcommand\qw{^{-1}}

\newcommand\qqw{^{-1/2}}

\newcommand{\pgf}{probability generating function}

\newcommand\lhs{left-hand side}
\newcommand\rhs{right-hand side}

\newcommand\xoo{_1^\infty}

\newcommand\cperm{C-permutation}
\newcommand\ctree{C-decorated tree}
\newcommand\fS{\mathfrak{S}}
\newcommand\fsc{\fS^{\textsf{C}}}
\newcommand\Nknm{N_{k;n.m}}
\newcommand\xinm{\xi^{(n,m)}}
\newcommand\Snm{S^{(n,m)}}
\newcommand\bm{\mathbf{m}}
\newcommand\syst{\mathrm{syst}}
\newcommand\NNng[1]{\bar{N}_{n,g_n}^{(#1)}}
\newcommand\kkx[1]{^{(#1)}}
\newcommand\kkk{\kkx{k}}
\newcommand\kkl{\kkx{\ell}}
\newcommand\bG{\mathbf{G}}
\newcommand\bGn{\bG_n}
\newcommand\bd{{\mathbf{d}}}
\newcommand\mnl{m\kkl_n}
\newcommand\sd{S_\bd}

\newcommand\PP{\mathfrak{P}}
\newcommand\PMx[1]{\PP^{[#1]}}
\newcommand\PM{\PMx{\bm}}
\newcommand\PPx[1]{\PP^{(#1)}}
\newcommand\PPl{\PPx{\ell}}
\newcommand\PPlx[1]{\PPx{\ell_{#1}}}
\newcommand\hPPx[1]{\widehat{\PP}^{(#1)}}
\newcommand\hPPl{\hPPx{\ell}}
\newcommand\hPPlx[1]{\hPPx{\ell_{#1}}}
\newcommand\pMx[1]{P^{[#1]}}
\newcommand\pM{\pMx{\bm}}

\newcommand\CC{\mathfrak{C}}
\newcommand\cc{C}
\newcommand\CCx[1]{\CC^{(#1)}}
\newcommand\CCl{\CCx{\ell}}

\newcommand\tCCx[1]{\widetilde\CC^{(#1)}}
\newcommand\tCCl{\tCCx{\ell}}
\newcommand\tCClx[1]{\tCCx{\ell_{#1}}}
\newcommand\ccx[1]{\cc^{(#1)}}
\newcommand\ccl{\ccx{\ell}}

\newcommand\tccx[1]{\widetilde\cc^{(#1)}}
\newcommand\tccl{\tccx{\ell}}
\newcommand\tcclx[1]{\tccx{\ell_{#1}}}
\newcommand\bP{\mathbf{P}}
\newcommand\Poo{\PP}
\newcommand\bt{{\mathbf t}}
\newcommand\simeqx{\equiv}

\newcommand\tpi{\tilde\pi}
\newcommand\On[1]{O\bigpar{n^{#1}}}
\newcommand\on[1]{o\bigpar{n^{#1}}}
\newcommand\bbt{\bar{\bt}}
\newcommand\Pm{\kappa^{(\bm)}}
\newcommand\Pmx[1]{\kappa^{(\bm(#1))}}
\newcommand\Npl{P_\ell}
\newcommand\tngs{(T_n,\bgs)}
\newcommand\mnu{\nu}
\newcommand\gq{\zeta}
\newcommand\SUCC{\textsf{succ}}
\newcommand\disj{\textsf{disj}}
\newcommand\Ext{\operatorname{Ext}}

\newcommand{\Holder}{H\"older}

\newcommand{\sig}{\boldsymbol{\sigma}}
\newcommand{\bgs}{\sig}
\newcommand{\T}{\mathbf{T}}

\begin{document}

\begin{abstract} 

We study large uniform random maps with one face whose genus grows linearly with
the number of edges, which are a model of discrete hyperbolic geometry. In previous works, several hyperbolic geometric features have been investigated.
In the present work, we study the number of short cycles in a uniform unicellular map of high genus, and we show that it converges to a Poisson distribution. As a corollary, we obtain the law of the systole of uniform unicellular maps in high genus. We also obtain the asymptotic distribution of the vertex degrees in such a map.
\end{abstract}

\maketitle

\section{Introduction}
\subsection*{Combinatorial maps.} Maps are defined as gluings of polygons forming a (compact, connected, oriented) surface. They have been studied extensively in the last decades, especially in the case of planar maps, i.e. maps of the sphere. They were first approached from the combinatorial point of view, both enumeratively, starting with \cite{Tut63}, and bijectively, starting with \cite{Sch98these}.

More recently, relying on previous combinatorial results, the properties of large random maps have been studied. More precisely, one can study the geometry of random maps picked uniformly in certain classes, as their size tend to infinity. In the case of planar maps, this culminated in the identification of two types of ``limits" (for two well defined topologies on the set of planar maps): the local limit (the \emph{UIPT}\footnote{in the case of triangulations, i.e. maps made out of triangles.} \cite{AS03}) and the scaling limit (the \emph{Brownian map} \cite{LG11, Mie11}).

All these works have been extended to maps with a fixed genus $g>0$. Enumerative (asymptotic) results have been obtained (see for instance \cite{BC86}), and there are bijections for maps on any surface (see for instance \cite{CMS09}). On the probabilistic side, equivalents of the Brownian map in genus $g>0$ have been constructed \cite{Bet16}.

\subsection*{High genus maps.} 
Very recently, another regime has been studied: high genus maps are defined as (sequences of) maps whose genus grow linearly in the size of the map. They have a negative average discrete curvature, and can therefore be considered as a discrete model of hyperbolic geometry.
It is usually hard to get geometric results on general models of maps (such as triangulations). So far, only the local convergence has been established \cite{BL19,BL20}, along with a few other geometric features \cite{Lou20}. 

However, there is a model that is easier to work with: unicellular maps,
i.e.\ maps with only one face. They are in bijection with decorated trees \cite{CFF13}, which makes them easier to study. Although they belong to a different universality class, it is widely expected that unicellular maps in high genus actually share very similar features with more general models of maps. So far, the local limit \cite{ACCR13}, the diameter \cite{Ray13a} and the presence of large expander subgraphs \cite{Lou21} of unicellular maps have been discovered (the last two properties are still open questions for general models of maps).

\subsection*{Distribution of short cycles.}
In this paper, we are interested in another geometric observable: the
\emph{systole}, i.e.\ the size of the smallest non-contractible cycles. The question of the systole is a well studied question in the field of geometry of surfaces (see for instance \cite{Gro92, Mir13, MP19} and in particular \cite{CVHM15} in the case of deterministic maps).

We
will work 
with unicellular maps, in which we have the extra advantage that any cycle
is non-contractible; hence the systole equals the \emph{girth}
(the length of the smallest cycle)
of the underlying graph. 
A little bit of progress on the question of short
non-contractible cycles in general maps is made in \cite{Lou20}. 

We actually prove a more general theorem, namely that the number of short
cycles in a high genus unicellular map follows a Poisson law.
For definitions, see \refS{Sdef}; 
note that the unicellular maps are rooted at an oriented edge, and
that we define the size of a map as its number of edges.

\begin{theorem}\label{thm_cycles}
Take a sequence $(g_n)$ such that $\frac{g_n}{n}\to \theta\in (0,1/2)$ as $n\to \infty$.
Let $C_{n,g_n}^{(k)}$ be the number of %simple ?? 
cycles of length $k$ in a
uniformly random unicellular map of size $n$ and genus $g_n$.
Then, for all $M$, we have
\begin{align}\label{tc}
(C_{n,g_n}^{(1)},C_{n,g_n}^{(2)},\ldots,C_{n,g_n}^{(M)})
\dto \Po(\lambda_1,\lambda_2,\ldots,\lambda_M),  
\end{align}
with
\begin{align}\label{gli}
  \gl_\ell :=\frac{1}{2\ell}\Bigpar{\Bigpar{\frac{1+\tau}{1-\tau}}^\ell
+ \Bigpar{\frac{1-\tau}{1+\tau}}^\ell-2},
\end{align}
where 
$\tau\in(0,1)$ is the (unique) root of
  \begin{align}\label{taug}
\frac{(1-\tau^2)\bigpar{\log(1+\tau)-\log(1-\tau)}}{2\tau}=1-2\gth.
  \end{align}
\end{theorem}
The parameter $\tau$ here is the same as $\gb$ in \cite{ACCR13}.
It is easily seen that $\tau$ increases from $0$ to $1$ as $\gth$ increases
from $0$ to $1/2$.

This theorem is an analogue 
of  the result of~\cite{MP19} for the Weil--Petersson measure (a classical continuous model of random hyperbolic geometry) as the genus tends to infinity.
As an immediate corollary, we obtain the law of the systole of high genus unicellular maps:

\begin{corollary}\label{C0}
Let $\syst(n,g_n)$ be the systole of a uniformly random unicellular
map of size $n$ and genus $g_n$. Then, as \ntoo{} with $g_n/n\to\gth\in(0,1/2)$,
% for all $\ell$
%\[\P(\syst(n,g_n))=\ell)\to e^{-\lambda_1}e^{-\lambda_2}\ldots
%e^{-\lambda_{\ell-1}}(1-e^{-\lambda_{\ell}}).\] 
\begin{align}\label{co1}
  \syst(n,g_n)\dto Z,
\end{align}
where $Z$ is a random variable with, letting the $\gl_i$ be as in \eqref{gli},
\begin{align}\label{coz}
\P(Z=\ell)= 
e^{-\lambda_1}e^{-\lambda_2}\dotsm
  e^{-\lambda_{\ell-1}}(1-e^{-\lambda_{\ell}}),
\qquad \ell\ge1.
\end{align}
\end{corollary}

It follows from \eqref{gli} that $\sum_\ell\gl_\ell=\infty$, 
and thus the probabilities in \eqref{coz} add up to $1$.
Therefore $Z$ is a proper distribution,
and, in particular, $\syst(n,g_n)$ is bounded in probability.

\subsection*{Vertex degrees.}
The methods we use also allow us to control the vertex degrees in the map:

\begin{theorem}\label{TD}
  Let $(g_n)$ be a sequence with $g_n/n\to\gth\in(0,1/2)$ as \ntoo.
Let $\NNng{k}$ be the number of vertices of degree $k\ge1$ in a uniformly
random unicellular map of size $n$ and genus $g_n$,
and let $m:=n+1-2g_n$ be the total number of vertices.
Then, as \ntoo,
\begin{align}\label{td}
  \NNng{k}/m\pto \kk_k,
\qquad k\ge1,
\end{align}
where
$(\kk_k)\xoo$ is a probability distribution given by 
\begin{align}
  \label{td2}
\kk_k=\frac{1}{\bigpar{\log(1+\tau)-\log(1-\tau)}k}
%\frac{1}{k}
\Bigpar{\Bigparfrac{1+\tau}{2}^k-\Bigparfrac{1-\tau}{2}^k},
\qquad k\ge1,
\end{align}
with $\tau$ given by \eqref{taug}.
\end{theorem}

\begin{remark}
\refT{TD} is related to  \cite[Theorem 4]{ACCR13}, which gives the
asymptotic distribution of the root degree of the map.

To see the connection, note first that
as a corollary of \refT{TD}, the degree of a uniformly
random vertex in the random  unicellular map converges in distribution to
the  distribution   $(\kk_k)$
given by \eqref{td2}. 
Since the number of edges and vertices in the map
are deterministic, this means that if we consider 
a uniformly random unicellular map of size and genus as above, but rooted at
a vertex instead of at an oriented edge, then the root degree distribution 
converges to $(\kk_k)$. 
It follows easily that for random unicellular maps rooted at an oriented edge, as in
\cite{ACCR13}  and the present paper, the root degree distribution converges
to the size biased 
distribution $(ck\kk_k)$, where $c$ is a normalization factor (easily computed);
this is \cite[Theorem 4]{ACCR13}. 
\end{remark}

\subsection*{Structure of the paper.} 
We end this section with a list of the main notations appearing in this paper. 
In section~\ref{Sdef}, we give definitions. 
Section~\ref{S:tree} states previous results on paths in trees that we use,
and Section~\ref{S:perm} is devoted to results about cycles in \cperm{s}.
In Section~\ref{S:ctree}, we prove
Theorem~\ref{thm_cycles} and finally Section~\ref{Sdegree} contains the
proof of Theorem~\ref{TD}. 
\subsection*{Index of notations}
This paper is quite notation heavy; we list here the main notations that appear throughout the paper.

\begin{itemize}
\item $g_n, \theta$: parameters of our model, $(g_n)_n\geq 1$ is the sequence of genuses of the unicellular maps we consider, with $\frac{g_n} n \to \theta\in(0,1/2)$;
\item $\tau, (\lambda_\ell)_{\ell\geq 1}$: parameters uniquely determined by $\theta$ (in~\eqref{taug} and~\eqref{gli});
\item $\cT_n$: the set of plane trees on $n$ edges;
\item $\T$: a random uniform element in $\cT_n$ (depends implicitly on $n$);
\item $\fsc_{n,m}$: the set of \cperm{s} on $n$ elements and $m$  cycles;
\item $\sig$: a random uniform element in $\fsc_{n+1,n+1-2g_n}$ (depends implicitly on $n$);
\item $T$: denotes a fixed rooted tree;
\item $\bt$: denotes a fixed unrooted tree;
\item $\Npl(T)$: number of paths of length $\ell$ in $T$;
\item $N_\bt(T)$: number of occurrences of $\bt$ in $T$;
\item $(T_n)_{n\geq 1}$: denotes a deterministic sequence of trees such that $T_n\in \cT_n$;
\item $C_{n,g_n}^{(k)}$: the number of %simple ?? 
cycles of length $k$ in a
uniformly random unicellular map of size $n$ and genus $g_n$ (alternatively, in the underlying graph of $(\T,\sig)$);
\item $\Ext$: will be used to denote the extremities of a path (or of a list of paths);

\item $\bP$: denotes a list of pairwise disjoint paths;
\item $s(\bP)$: number of paths in $\bP$;
\item $\ell(\bP)$: total length of the paths of $\bP$;
\item $\mathbf{m}$: denotes the list of lengths of a list of paths.

\end{itemize}
\section{Definitions}\label{Sdef}

A unicellular map of \emph{size} $n$ is a $2n$-gon whose sides were glued two by
two to form a (compact, connected, oriented) surface. The \emph{genus} of
the map is the genus of the surface created by the gluings (its number of
handles). After the gluing, the sides of the polygon become the \emph{edges}
of the map, and the vertices of the polygon become the \emph{vertices} of
the map. 
Note that the number of edges equals the size $n$.
By Euler's formula, a unicellular map of genus $g$ and size $n$ has $n+1-2g$
vertices. The underlying graph of a unicellular map is the graph obtained
from this map by only remembering its edges and vertices.
(In general, this is a multigraph.)

We consider in this paper only
\emph{rooted} unicellular maps, where an oriented edge is marked as the
\emph{root}. 
The underlying graph is then a rooted graph.

A rooted unicellular map of genus $0$ is the same as a plane tree
(i.e., an ordered rooted tree). 
We denote by $\cT_n$ the set of plane trees of size $n$ (and thus $n+1$
vertices). 

A \emph{path} $\cP$ of length $\ell\ge0$ in a (multi)graph
is a list $(v_0,e_1,v_1,e_2,\ldots,e_\ell,v_\ell)$ of 
alternating distinct\footnote{with this definition, such objects are also
  commonly called simple paths.} 
vertices and edges such that for all $1\leq i\leq \ell$, $v_{i-1}$ and $v_i$
are joined by the edge $e_i$. 
We define $\start(\cP):=v_0$ and $\eend(\cP):=v_\ell$, and let
$\Ext(\cP):=\set{\start(\cP),\eend(\cP)}$. 
Similarly, a \emph{cycle} of length $\ell\ge1$ is 
a list $(v_0,e_1,v_1,e_2,\ldots,e_\ell)$ of distinct vertices and edges
such that
$(v_0,e_1,v_1,e_2,\ldots,v_{\ell-1})$ is a path and $e_\ell$ is an edge
joining $v_{\ell-1}$ to $v_0$.
In a simple graph (for example, a tree), the edges $e_i$ in a path or cycle
are determined by the vertices, so it suffices to specify $(v_0,v_1,\dots)$.
We denote the length of a path $p$ by $|p|$; this is thus the number of
edges in $p$.

A \emph{\cperm{}} is a permutation whose cycles are of odd length.
Let $\fsc_n$ be the set of \cperm{s} of length $n$, and $\fsc_{n,m}$ the
subset of permutations in $\fsc_n$ with exactly $m$ cycles.
(This is empty unless $n\equiv m \pmod 2$; we assume tacitly in the sequel
that we only consider cases with $\fsc_{n,m}\neq\emptyset$.)
Note that our definition of a \cperm{} differs from the one given in \cite{CFF13}, where each cycle carries an additional sign. Here we do not include the signs as they will not play a role in our proofs.

A \emph{\ctree{}} of size $n$ and genus $g$ is a pair $(T,\sigma)\in \cT_n\times \fsc_{n+1,n+1-2g}$ where $\sigma$ is seen as a \cperm{} of the vertices of $T$. The underlying graph of $(T,\sigma)$ is the graph obtained by merging the vertices of $T$ that belong to the same cycle in $\sigma$. If $v,v'\in T$, we write $v\sim v'$ if $v$ and $v'$ belong to the same cycle in $\sigma$.

\begin{theorem}[\cite{CFF13}, Theorem 5]\label{thm_ctrees}
Unicellular maps of size $n$ and genus $g$ are in $2^{2g}$ to $1$ correspondence  with \ctree{s} of size $n$ and genus $g$. This correspondence preserves the underlying graph.
\end{theorem}

Therefore, with this correspondence, it is sufficient to study \ctree{s}.

More precisely, from now on, we fix a number $\theta \in (0,1/2)$ and a
sequence $(g_n)_{n\geq 1}$ such that $\frac{g_n}{n}\to \theta$.

Now, we define (for each $n$, with $n$ only implicit in the notation) 
$\T$ to be a uniformly random tree in $\cT_n$, and $\sig$
to be a uniformly random C-permutation  in $\fsc_{n+1,n+1-2g_n}$,
with $\T$ and $\sig$ independent. From now on,
thanks to Theorem~\ref{thm_ctrees}, we can assume that $C_{n,g_n}^{(k)}$
is the number of cycles of size $k$ in the underlying graph of $(\T,\sig)$.

\subsection{Further notation}\label{Snot}

The multivariate Poisson distribution $\Po(\gl_1,\dots,\gl_M)$, 
%where $\gl_i\ge0$,
is the distribution of a random vector $(X_1,\dots,X_M)$ with independent
Poisson distributed components $X_i\sim\Po(\gl_i)$.

$(n)_r$ denotes the descending factorial $n(n-1)\dotsm(n-r+1)$.

$C$ and $c$ denote unspecified constants that may vary from one occurrence
to the next. $C_r$ denotes a constant depending on $r$.

Unspecified limits are as \ntoo.

\section{Paths in trees}\label{S:tree}

The following lemma helps us estimate the number of paths of a certain
length in the random tree $\T$.
In general, for a rooted tree $T$,
we let $\Npl(T)$ be the number of  paths of length $\ell$ in $T$.

\begin{lemma}[\cite{SJ355}, Example 4.3, (4.9)]\label{lem_paths}
Let, as above, $\T$ be a uniformly random tree in $\cT_n$.
Then, as \ntoo,
for every fixed $\ell\ge1$, 
%let $P_{\ell}^{(n)}$ be the number of paths of length $\ell$ in $\T$. Then 
\begin{equation}\label{lep1}
\frac{\Npl(\T)}{n}\pto 2\ell.
\end{equation}
\end{lemma}

To estimate the numbers of intersecting paths, we will also use the following,
more general, result.
(We count paths as oriented, but general subtrees $\bt$ as unlabelled.)

\begin{lemma}[\cite{SJ355}]\label{LNt}
Let, as above, $\T$ be a uniformly random tree in $\cT_n$.
Let $\bt$ be a fixed unrooted tree, and let $N_\bt(\T)$ be the number of
subtrees of $\T$ isomorphic to $\bt$.
Then, as \ntoo,
\begin{equation}\label{lep2}
\frac{N_\bt(\T)}{n}\pto c_\bt,
\end{equation}
for some constant $c_\bt\in(0,\infty)$.
\end{lemma}

\section{Cycles in \cperm{s}}\label{S:perm}

  Let $\gs_{n,m}$ be a uniformly random element of $\fsc_{n,m}$, 
and let $\Nknm$ be its number of cycles of length $k$.
Thus $\Nknm=0$ unless $k\ge1$ is odd.
Note also that
\begin{align}\label{eq_constraint_cycles}
 \sum_k \Nknm = m, &&&
\sum_k k\Nknm = n.
\end{align}

\begin{proposition}\label{T1}
Let $0<\ga<1$ and suppose that $m,n\to\infty$ with $m/n\to\ga$.
Then, for each $k\ge1$,
\begin{align}\label{t1}
  \frac{\Nknm}m \pto \pi_k,
\end{align}
where $(\pi_k)$ is a probability distribution 
supported
on \set{1,3,5,\dots}
given by
\begin{align}\label{pi}
  \pi_k:=
  \begin{cases}
    \frac{1}{\Phi(\tau)}\frac{\tau^k}{k}, & \text{$k$ odd},
\\
0,& \text{$k$ even},
  \end{cases}
\end{align}
where $\tau\in(0,1)$ is the (unique) root of
  \begin{align}\label{tau}
\frac{(1-\tau^2)\bigpar{\log(1+\tau)-\log(1-\tau)}}{2\tau}=\ga
  \end{align}
and $\Phi(\tau)$ is a normalizing constant given by \eqref{Phi} below.

Furthermore, for any $r\ge0$,
\begin{align}\label{surr}
  \frac{1}{m}\sum_k k^r \Nknm \pto \sum_k k^r \pi_k,
\end{align}
with convergence of all moments.

Finally, there exists a constant $C_\alpha$ such that, if we define the event 
\begin{equation}\label{eq_PE}
\mathcal{E}:=
\set{\text{The largest cycle of $\gs_{n,m}$ is shorter than $C_\alpha\log^2 n$}}
\end{equation}
then 
\begin{equation}
\P(\mathcal{E})\geq 1-n^{-\log n}.
\end{equation}

\end{proposition}
\begin{proof}
  Let $L_1,\dots,L_m$ be the lengths of the cycles of $\gs_{n,m}$, taken in
  (uniformly) random order.
Note that
\begin{align}
  \label{suml}
\sumim L_i=n.
\end{align}

We say that a permutation with $m$ cycles has labeled cycles if we have
labeled its cycles by $1,\dots,m$.

For any positive integers $\ell_1,\dots,\ell_m$ with $\sum_i\ell_i=n$, the
number of permutations of length $n$ with $m$ labelled cycles with lengths
$\ell_1,\dots,\ell_m$, respectively, is
\begin{align}
  \binom{n}{\ell_1,\dots,\ell_m}\prod_{i=1}^m (\ell_i-1)!
=
n!\prod_{i=1}^m \frac{1}{\ell_i}.
\end{align}
Consequently,
\begin{align}
  \P\bigpar{(L_1,\dots,L_m)=(\ell_1,\dots,\ell_m)}
= \frac{1}{Z}\prod_{i=1}^m \frac{1}{\ell_i}\indic{\ell_i \text{ is odd}},
\end{align}
where $Z=Z\xpar{n,m}$ is a normalization factor.

This means that $(L_1,\dots,L_m)$ is an instance of a well-known random
allocation model, called \emph{balls-in-boxes} in \cite[Section 11]{SJ264},
with weights
\begin{align}\label{w}
  w_\ell :=\frac{1}{\ell}\indic{\ell \text{ is odd}},
\qquad \ell\ge0.
\end{align}
(Warning: $n$ and $m$ have opposite meanings in \cite{SJ264}.)

In the notation of 
%\cite[in particular Section 3]{SJ264}, 
\cite[(3.4)]{SJ264}, 
we have the
generating function 
\begin{align}\label{Phi}
  \Phi(z)&:=\sumlo w_\ell z^\ell 
=\sumio \frac{z^{2i+1}}{2i+1}
\notag\\&\phantom:
= \frac12\bigpar{\log(1+z) - \log(1-z)}
= \frac12\log\frac{1+z}{1-z}
\end{align}
and, see \cite[(3.1), (3.5), (3.6), (3.10)]{SJ264},
\begin{align}
\label{go}
  \go&=\infty,\\
\label{rho}
\rho&=1,\\
\label{Psi}
\Psi(t)&=\frac{t\Phi'(t)}{\Phi(t)}
=\frac{2t}{(1-t^2)\bigpar{\log(1+t)-\log(1-t)}},\\
\label{nu}
\nu&:=\lim_{t\upto\rho}\Psi(t)=\infty
.\end{align}

We apply \cite[Theorem 11.4]{SJ264} with $\gl=1/\ga>1$; 
this theorem is stated assuming
$w_0>0$, but we may apply the theorem to the random allocation
$(L_1-1,\dots,L_m-1)$, noting that \eqref{suml} is equivalent to
\begin{align}
  \label{sumli}
\sumim (L_i-1)=n-m,
\end{align}
and it is easily seen that the conclusions of \cite[Theorem 11.4]{SJ264}
hold for any $\gl\ge1$. 
This yields \eqref{t1}; note that \eqref{tau} is equivalent to
$\Psi(\tau)=1/\ga=\gl$. 

To show \eqref{surr}, we first show that the expectations are bounded:
for every $r\ge0$,
\begin{align}\label{suc}
\E  \frac{1}{m}\sum_k k^r \Nknm 
=
\sum_k k^r \E \frac{\Nknm}{m}
\le C_r.
\end{align}
In fact, we will show the stronger estimate that if $1<A<\tau\qw$, then
\begin{align}\label{a1}
%  \sum_k A^k \E \frac{\Nknm}{m}\le C.
  \E \frac{\Nknm}{m}\le C A^{-k}, \qquad k\ge1.
\end{align}
(Recall that $\tau\qw>1$.)

To see \eqref{a1}, we use arguments similar to those in \cite[Section
14]{SJ264}. Note that, by symmetry,
\begin{align}\label{a0}
\E \frac{\Nknm}{m} = \P(L_1=k).
\end{align}
 Let $\tau_{n,m}$ be given by $\Psi(\tau_{n,m})= n/m$. Since $n/m\to\gl$, it
 follows that $\tau_{n,m}\to\tau<A\qw$.
Thus, $\tau_{n,m}<A\qw$ 
when $n$ and $m$ are large enough; consider only such $n$
(and $m$).

Given such $n$ and $m$, let $\xinm_i$, $i\ge1$, be \iid{} random variables
with the distribution
\begin{align}\label{a2}
  \P\bigpar{\xinm_i=k} 
= \frac{1}{\Phi(\tau_{n,m})}\frac{\tau_{n,m}^k}k\indic{k \text{ is odd}},
\qquad k\ge1,
\end{align}
and
let $\Snm_j:=\sum_{i=1}^j\xinm_i$.
The random vector $(L_1,\dots,L_m)$ has the same distribution as 
$(\xinm_1,\dots,\xinm_m)$ conditioned on $\Snm_m=n$,
and it follows that
\begin{align}\label{a3}
  \P(L_1=k)
= 
\frac{ \P\bigpar{\xinm_i=k} \P\bigpar{\Snm_{m-1}=n-k}}{\P\bigpar{\Snm_m=n}}.
\end{align}
It follows from \cite[Lemma 14.1]{SJ264} 
(applied to $\xinm_i-1$)
that
${\P\bigpar{\Snm_m=n}} \sim c m\qqw$.
Similarly,
it follows from \cite[Remark 14.2]{SJ264} that
$ \P\bigpar{\Snm_{m-1}=n-k}\le C m\qqw$. 
Hence \eqref{a3} yields, using
\eqref{a2}, 
\begin{align}\label{a4}
  \P(L_1=k)
\le C \P\bigpar{\xinm_i=k} 
\le C \tau_{n,m}^k 
\le C A^{-k}.
\end{align}
Consequently, \eqref{a1} follows by \eqref{a0} and \eqref{a4},
and then \eqref{suc} follows.

Next,
we have for any $K<\infty$,
\begin{align}\label{a5}
  \E\,& \Bigabs{  \frac{1}{m}\sum_k k^r \Nknm - \sum_k k^r \pi_k}
\notag\\&
\le
\sum_{k=1}^Kk^r  \E \Bigabs{\frac{\Nknm}m -  \pi_k}
+     \sum_{k>K} k^r\E\frac{\Nknm}m + \sum_{k>K} k^r \pi_k.
\end{align}
Given any $\eps>0$, we can by \eqref{a1} and \eqref{pi} 
choose $K$ so large that the two last sums each
are less than $\eps$ for all large $n$; furthermore, each term in the first
sum is $o(1)$ by \eqref{t1} and bounded convergence.
It follows that the \lhs{} of \eqref{a5} is $o(1)$, and thus \eqref{surr}
holds.

Finally, for any $p\ge1$, by \Holder's inequality,  $\sum_k \Nknm/m=1$ and
\eqref{suc}, 
\begin{align}
  \E \Bigpar{\frac{1}{m} \sum_k k^r\E \Nknm}^p
\le
  \E  \sum_k k^{pr} \frac{\Nknm}m
\le C_{pr}.
\end{align}
Hence, all moments are bounded in \eqref{surr}, which implies uniform
integrability of any fixed power and thus convergence of
all moments, see \eg{} \cite[Theorems 5.4.2 and 5.5.4]{Gut}.

To prove \eqref{eq_PE}, notice that for $C_\alpha$ large enough, we have by \eqref{a4},
\begin{align}
  \P(L_1\geq C_\alpha \log^2 n)
\le \frac{1}{n}n^{-\log n},
\end{align}
and we conclude by a union bound, since the number of cycles is  $m\le n$.
\end{proof}

Let $M:=M_{n,m}$ be the number of (unordered) pairs of elements of $[n]$
that belong to the same cycle in $\gs_{n,m}$, and thus are merged together
when cycles are merged. Then
\begin{align}\label{M}
  M_{n,m}=\sum_k \binom k2 \Nknm.
\end{align}

Using the crude bound $\binom k 2 \geq k-1$
and~\eqref{eq_constraint_cycles}, 
we obtain the following deterministic bound: 
\begin{equation}\label{eq_deterministic_bound_pairs}
M_{n,m}\geq n-m.
\end{equation}

The following lemma gives a better estimation of $M_{n,m}$. 
\begin{lemma}\label{LQC1}
With assumptions and notations as in Proposition~\ref{T1},
\begin{align}\label{c1}
\frac{M_{n,m}}{m}
&\pto  
\sum_k \binom{k}2\pi_k 
%\notag\\&
=\frac{2\tau^3}{(1-\tau^2)^2\bigpar{\log(1+\tau)-\log(1-\tau)}}
\notag\\&
=\frac{\tau^2}{\ga(1-\tau^2)}
.\end{align}
Equivalently,
\begin{align}\label{c1b}
\frac{M_{n,m}}{n}
&\pto  
\gam:=\frac{\tau^2}{1-\tau^2}
.\end{align}
\end{lemma}

\begin{proof}
  The convergence in probability and with all moments follows by \eqref{M}
  and \eqref{surr}.
To compute the sum, note that if $\xi$ is a random variable with distribution
$(\pi_k)_k$, then its \pgf{} is
\begin{align}
  \gf_X(z):=\sum_k \pi_k z^k
=\frac{\Phi(\tau z)}{\Phi(\tau)}.
\end{align}
Hence,
\begin{align}
  \sum_k \binom{k}2\pi_k 
&= \E\frac{X(X-1)}2
= \tfrac12\gf_X''(1)
=\frac{\tau^2\Phi''(\tau)}{2\Phi(\tau)}
\notag\\&
=\frac{2\tau^3}{(1-\tau^2)^2\bigpar{\log(1+\tau)-\log(1-\tau)}}
,\end{align}
and the result follows using \eqref{tau}.
\end{proof}

The results above hold in full generality, but here we go back to our
setting. We recall that $\frac{g_n}{n}\to \theta$ and we set
$\alpha=1-2\theta$ (and therefore the constant $\gamma$ now depends on
$\theta$); thus $m/n=(n+1-2g_n)/n\to\ga$.

\begin{lemma}\label{lem_proba_merging}
Let $\cE_{n,g_n}^{(r)}$ be the event where, in $\sig$, $2i-1$ and $2i$
belong to the same cycle for all $1\leq i\leq r$, and that all these cycles
are distinct, and let $\P_{n,g_n}^{(r)}=\P(\cE_{n,g_n}^{(r)})$. Then, for
any fixed $r$,
\begin{equation}\label{er}
\P_{n,g_n}^{(r)}=(1-o(1))\left(\frac{2\gamma}{n}\right)^r.
\end{equation}
\end{lemma}

\begin{proof}
By \eqref{eq_PE}, we have 
\begin{equation}\label{eq_cE}
\P_{n,g_n}^{(r)}=\P(\cE_{n,g_n}^{(r)}\mid\cE)+o(n^{-r}).
\end{equation}
In what follows, we will reason conditionally on $\cE$. 
We let $\Delta$ be the size of the biggest cycle in $\sig$.
Thus, by the definition \eqref{eq_PE}, $\cE$ is the event
\begin{equation}\label{eq_Delta}
\Delta \leq C_\alpha \log^2 n.
\end{equation}

The number of $r$-tuples of (unordered, non overlapping) pairs of elements of $[n+1]$
that belong to the same cycle in $\sig$ is at most
$(M_{n+1,n+1-2g_n})^r$ (because this quantity counts $r$-tuples with possible overlaps). On the other hand, we can lower bound the number of $r$-tuples of pairs of elements of $[n+1]$ belonging to $r$ distinct cycles by
\begin{equation}
\prod_{i=0}^{r-1} \left(M_{n+1,n+1-2g_n}-i\binom \Delta 2\right).
\end{equation}

By \eqref{eq_deterministic_bound_pairs} and \eqref{eq_Delta}, we have 
\begin{equation}
\prod_{i=0}^{r-1} \left(M_{n+1,n+1-2g_n}-i\binom \Delta 2\right)=(1-o(1))(M_{n+1,n+1-2g_n})^r.
\end{equation}

The total number of $r$-tuples of (non-overlapping) pairs of elements of
$[n]$ is $\binom n {2,2,\ldots,2,n-2r}$. 
Therefore,
\begin{equation}
\P(\cE_{n,g_n}^{(r)}\mid\cE,M_{n+1,n+1-2g_n})=(1-o(1))\frac{(M_{n+1,n+1-2g_n})^r}{\binom {n+1} {2,2,\ldots,2,n+1-2r}}.
\end{equation}

Finally, using \eqref{eq_cE} and the convergence in probability established in~\eqref{c1b}, we can conclude that
\begin{equation}
\P_{n,g_n}^{(r)}=(1-o(1))\left(\frac{2\gamma}{n}\right)^r.
\end{equation}
\end{proof}

\begin{lemma}\label{lem_proba_merging_hard}
Let $\P_{n,g_n}^{(r,k)}$ be the probability that, in $\sig$, 
there exist $r+k$ distinct cycles $C_1,C_2,\ldots,C_r,\tilde{C_1},\tilde{C_2},\ldots,\tilde{C_k}$ such that
 $2i-1$ and $2i$ belong $C_i$ for all $1\leq i\leq r$, and $2r+i$ belongs to
 $\tilde{C_i}$ for all $1\leq i\leq k $ . Then, for every fixed $r$ and $k$,
\begin{equation}
\P_{n,g_n}^{(r,k)}=(1-o(1))\P_{n,g_n}^{(r)}.
\end{equation}
\end{lemma}

\begin{proof}

Let $\Delta$ be the maximal size of a cycle in $\sig$. By the same reasoning
as in the proof of Lemma~\ref{lem_proba_merging}, we can work conditionally
on $\gD$, and assume
\begin{equation}\label{eq_Delta2}
\Delta \leq C_\alpha \log^2 n.
\end{equation}

Condition also 
on the fact that there exists $r$ distinct cycles $C_1,C_2,\ldots,C_r$ of
$\sig$ such that $2i-1$ and $2i$ belong to $C_i$ for all $1\leq i\leq r$
(which amounts to conditioning on $\cE_{n,g_n}^{(r)}$). We want to show that
the rest of the constraint holds with high probability. Note that the
numbers $2r+1\leq i\leq 2r+k$ are distributed uniformly in the cycles of
$\sig$ among all remaining spots. Therefore, given $2r+1\leq j\leq 2r+k$,
the probability that $j$ belongs to a given $C_i$ is less than
$\frac{\Delta}{n+1-2r-k}$. By the same reasoning, given $2r+1\leq i<j\leq
2r+k$, the probability that $i$ and $j$ belong to the same cycle is also
bounded by 
$\frac{\Delta}{n+1-2r-k}$. 

By a union bound, the conditional probability of failure is bounded by
\begin{equation}
\left(rk+\binom k2\right)\frac{\Delta}{n+1-2r-k},
\end{equation}
which is $o(1)$ by \eqref{eq_Delta2}.
\end{proof}
\begin{lemma}\label{lem_points_in_cycles}
For fixed $k$ and $t<k$, the probability that $1,2,\ldots,k$ belong to
exactly $t$ distinct cycles of $\sig$ is
\begin{align}
  \label{sof1}
O\left(\frac{1}{n^{k-t-1/4}}\right),
\end{align}
\end{lemma}

\begin{proof}
Let $\Delta$ be the maximal size of a cycle in $\sig$. By the same reasoning
as in the proof of Lemma~\ref{lem_proba_merging}, we can work conditionally
on 
$\gD$ and assume
\begin{equation}\label{eq_Delta3}
\Delta \leq C_\alpha \log^2 n.
\end{equation}

Let us call $\SUCC$ the event that $1,2,\ldots,k$ belong to exactly $t$
distinct cycles of $\sig$.
We will prove a stronger result, namely that,
conditioned on $\gD$,
\begin{align}
  \label{sof2}
\P(\SUCC)=O\left(\left(\frac{\Delta}{n}\right)^{k-t}\right),
\end{align}
and we conclude using \eqref{eq_Delta3}.

Let $\disj$ be the event that $1,2,\ldots,t$ belong to exactly $t$ distinct cycles of $\sig$.

By a union bound and symmetry, we have
\begin{equation}\label{eq_proba_condi}
\P(\SUCC)\leq \binom kt \P(\SUCC\mid \disj)\P(\disj)
\leq \binom kt \P(\SUCC\mid \disj).
\end{equation}
Now, conditioning also on $\disj$, the elements $t+1,\ldots,k$ are placed
uniformly in the remaining spots of $\sig$. The total size of the cycles
containing $1,2,\ldots,t$ is less than $t\Delta$; hence,
\begin{equation}\label{eq_Psucc_condi}
\P(\SUCC\mid \disj)\leq \left(\frac{t\Delta}{n+1-k}\right)^{k-t}.
\end{equation}
%
%Finally, by Lemma~\ref{lem_proba_merging_hard} (for $r=0$ and $k$ replaced by $k-t$), we have
%\begin{equation}\label{eq_Pdisj}
%\P(\disj)=1-o(1).
%\end{equation}
Combining \eqref{eq_proba_condi} and  \eqref{eq_Psucc_condi}  
yields \eqref{sof2} and thus \eqref{sof1}.
\end{proof}

\section{Cycles in \ctree{s}}\label{S:ctree}

If $(T,\sigma)$ is a \ctree, then any simple cycle of length $\ell$ in its
underlying graph can be decomposed into a list $\bP=(p_1,p_2,\ldots,p_k)$ 
of non-intersecting simple paths in $T$ 
such that 
\begin{PXenumerate}{C}
\item\label{C1} $\sum_{i=1}^k |p_i|=\ell$,
\item\label{C2} $\eend(p_i)\sim \start(p_{(i+1 \mod k)})$ for all $i$,
\item \label{C3} 
for every other pair of vertices $v,v'\in (p_1,p_2,\ldots,p_k)$, we
  have $v\not\sim v'$. 
\end{PXenumerate}
This decomposition is unique up to cyclically reordering the $p_i$, or
reversing them all and their order, or a combination of both.

%Hence, we introduce the following  notation.

\subsection{More notation}\label{SSmore}
Let $T$ be a rooted tree.

Let $\Poo(T)$ be the set of all lists $\bP=(p_1,\dots,p_k)$ of pairwise
disjoint paths in $T$, of arbitrary length $k\ge1$.
For a list $\bP=(p_1,\dots,p_k)\in\Poo(T)$, let $s(\bP):=k$, the number of
paths in the list, and $\ell(\bP):=\sum_1^k|p_i|$, their total length.
%(Recall that $|p_i|$ is the number of edges in $p_i$.)
Also, let
$\Ext(\bP):=\bigcup_i\Ext(p_i)=\set{\start(p_i),\eend(p_i):i=1,\dots,k}$, the
set of endpoints of the paths in $\bP$; note that $|\Ext(\bP)|=2s(\bP)$.

For an integer $\ell\ge1$, let
\begin{align}\label{PPl}
\PPl(T):=\set{\bP\in\Poo(T):\ell(\bP)=\ell},   
\end{align}
the set of lists of paths with
total length $\ell$.
Furthermore, 
for any finite sequence of positive integers $\mathbf{m}=(m_1,\ldots,m_k)$,
we write $|\mathbf{m}|=m_1+m_2+\ldots+m_k$ and $s(\mathbf{m})=k$. 
We define 
\begin{align}\label{PM}
\PM(T):=\set{\bP=(p_1,\dots,p_k)\in\Poo(T): |p_i|=m_i \forall i}, 
\end{align}
the set of lists of pairwise disjoint  paths
$(p_1,\ldots,p_k)$ of lengths $m_1,\ldots,\allowbreak m_k$, respectively, 
in $T$.
We let $\pM(T):=|\PM(T)|$ be its cardinal.
Define also
\begin{equation}\label{eq_Pm}
\Pm=\prod_{i=1}^{s(\mathbf{m})} 2m_i.
\end{equation}

For two lists $\bP,\bP'\in\Poo(T)$, 
we write $\bP\simeqx \bP'$ if and only if $\bP'$ can be obtained from 
$\bP$ by cyclically reordering its paths, or
reversing them all and their order, or a combination of both.
Note that $\bP\simeqx \bP'$ entails $s(\bP)=s(\bP')$ and
$\ell(\bP)=\ell(\bP')$, and that each list $\bP$ is in an equivalence class
$[\bP]$ with exactly $2s(\bP)$ elements.
Let $\hPPl(T)$ be a subset of $\PPl(T)$ obtained by selecting exactly one
element from each equivalence class in $\PPl(T)$.

Given also a \cperm{} $\gs$ of the vertex set of $T$, 
so that $(T,\gs)$ is a \ctree,
let $\CCl(T,\gs)$ be the set of lists $\bP=(p_1,\dots,p_k)\in\hPPl(T)$
that satisfy \ref{C1}--\ref{C3} above.
Thus, there is a bijection between $\CCl(T,\gs)$
and the set of  cycles of length $\ell$
in the underlying graph of the \ctree{} $(T,\gs)$.
Hence, $\ccl(T_n,\gs):=|\CCl(T_n,\gs)|$ equals the number of cycles of
length $\ell$ in the underlying graph.

Furthermore, 
let $\tCCl(T,\gs)$ be the set of lists $\bP=(p_1,\dots,p_k)\in\hPPl(T)$
that satisfy \ref{C1}--\ref{C2}.
Thus $\tCCl(T,\gs)\supseteq\CCl(T,\gs)$
Let further 
$\tccl(T_n,\gs):=|\tCCl(T_n,\gs)|$,
and note that $\tccl(T,\gs)\ge\ccl(T,\gs)$.

As in \refS{S:tree},
let $\Npl(T)$ be the number of  paths of length $\ell$ in $T$. 
Similarly, 
if $\bt$ is an unrooted tree, let $N_\bt(T)$ be the number of subtrees of
$T$ that are isomorphic to $\bt$. 

\subsection{Lemmas and proof of \refT{thm_cycles}}
We are really interested in random trees $\T$, but it will be convenient to
first consider non-random trees.
We suppose for the following lemmas
that $(T_n)$ is a given sequence of
(deterministic) rooted trees, with $T_n\in\cT_n$.
We say that the sequence is \emph{good}, if, as \ntoo,
\begin{align}\label{apl}
  \Npl(T_n)/n \to 2\ell, \qquad \forall\ell\ge1,
\end{align}
and, for every finite unrooted tree $\bt$, 
there is a constant $c_\bt\ge0$ such that
\begin{align}\label{ant}
  N_\bt(T_n)/n\to c_\bt.
\end{align}

\begin{lemma}\label{LC1}
  Let $(T_n)$ be a good sequence of trees.
Then, for every $\bm=(m_1,\dots,m_k)$,
\begin{align}\label{lc1}
  \pM(T_n)=|\PM(T_n)|
=\Pm n^k +\on{k}.
\end{align}
\end{lemma}

\begin{proof}
  An element of $\PM$ is a  sequence of disjoint paths $p_1,\dots,p_k$ in
  $T_n$, with $|p_i|=m_i$.
Thus, if we ignore the condition that the paths be disjoint, 
each $p_i$ may be chosen in $P_{m_i}(T_n)$ ways, which  gives
$\prod_1^k P_{m_i}(T_n)$ sequences.
We have to subtract the number of sequences with two intersecting paths
$p_i$ and $p_j$. It suffices to consider the case when $p_1$ and $p_2$
intersect. Then 
their union is a subtree $\bt$ of $T_n$ with at most $m_1+m_2+1$ vertices.
Regard $\bt$ as an unlabelled tree $\bbt$.  There is only a finite
number of possible such $\bbt$, and for each choice, the condition
\eqref{ant} shows that the number of possible subtrees $\bt$ in $T_n$ 
is $O(n)$.
Moreover, for each $\bt$, the paths $p_1$ and $p_2$ are subsets of $\bt$
and can thus be chosen in $O(1)$ ways. The remaining paths $p_3,\dots,p_k$
can, as above, each be chosen in $O(n)$ ways.
Consequently, the number of sequences of paths $(p_1,\dots,p_k)$ such that
$|p_i|=m_i$ and two of the paths intersect is $\On{k-1}$.
It follows that, using \eqref{apl} and \eqref{eq_Pm},
\begin{align}\label{apa}
%  \pM
|\PM(T_n)|&
=\prodik P_{m_i}(T_n) - \On{k-1}
=\prodik \bigpar{\bigpar{2m_i+o(1)}n} -\On{k-1}
\notag\\&
%=\prodik \xpar{2m_i}\cdot n^k +\on{k}.
=\Pm n^k +\on{k},
\end{align}
which proves \eqref{lc1}
\end{proof}

In the remainder of this section, $\bgs=\bgs_n$ is a uniformly random
\cperm{} in  $\fsc_{n+1,n+1-2g_n}$.

We begin by calculating the expectation of $\ccl(T_n,\bgs)$.
\begin{lemma}\label{LC2}
  Let $(T_n)$ be a good sequence of trees.
Then, for every $\ell\ge1$, as \ntoo,
\begin{align}\label{lc2a}
  \E \ccl(T_n,\bgs) \to\gl_\ell
\end{align}
given by \eqref{gli}.
Furthermore,
\begin{align}\label{lc2b}
\E\tccl(T_n,\bgs)-  \E \ccl(T_n,\bgs) \to 0.
\end{align}
\end{lemma}

\begin{proof}
  We begin with the simpler $\tccl(T_n,\bgs)=|\tCCl(T_n,\bgs)|$.
For each given list $\bP=(p_1,\dots,p_k)\in\hPPl(T_n)$,
let $\tpi(\bP)$ be the probability that $\bP\in \tCCl(T_n,\bgs)$,
in other words, the probability that \ref{C2} holds. % for  $(T_n,\bgs)$.
Then, by definitions and symmetry,
\begin{align}\label{lc2c}
  \E\tccl(T_n,\bgs) 
= \sum_{\bP\in\hPPl(T_n)}\tpi(\bP)
=\sum_{\bP\in\PPl(T_n)}\frac{\tpi(\bP)}{2s(\bP)}.
\end{align}

To find $\tpi(\bP)$,
we may relabel the $2k$ endpoints in $\Ext(\bP)$ as $1,\dots,2k$ 
in an order such  that \ref{C2} becomes $2i-1\sim 2i$ for $i=1,\dots,k$,
\ie, that $2i-1$ and $i$ belong to the same cycle in $\bgs$.
There are two cases: either these $k$ cycles are distinct, or at least two
of them coincide.
The first event is 
$\cE_{n,g_n}^{(k)}$ in \refL{lem_proba_merging}, and that lemma shows that
its probability is
\begin{equation}\label{ERII}
\P_{n,g_n}^{(k)}=
\P(\cE_{n,g_n}^{(k)})
=(1-o(1))\xpar{2\gamma}^kn^{-k}.
\end{equation}
The second event means that $1,\dots,2k$ belong to at most $k-1$ different
cycles of $\bgs$. Hence, \refL{lem_points_in_cycles} shows that the
probability of this event is
\begin{align}\label{lc2d}
  \sum_{t=1}^{k-1}O\bigpar{n^{t-2k+1/4}}
=O\bigpar{n^{-k-3/4}}
=o\bigpar{n^{-k}}.
\end{align}
Summing \eqref{ERII} and \eqref{lc2d}, we see that 
\begin{align}\label{pi1}
  \tpi(\bP)=\bigpar{(2\gam)^k+o(1)} n^{-k},
\end{align}
where $k=s(\bP)$.

We develop the sum in \eqref{lc2c} 
using \eqref{pi1} and \eqref{lc1}, %\eqref{eq_Pm}, 
and obtain
\begin{align}\label{pia}
  \E\tccl(T_n,\bgs) &
=\sum_{|\bm|=\ell} |\PM(T_n)|\frac{1}{2s(\bm)} 
\bigpar{(2\gam)^{s(\bm)}+o(1)} n^{-s(\bm)}
\notag\\&
=\sum_{|\bm|=\ell} \frac{\Pm}{2s(\bm)}(2\gam)^{s(\bm)}+o(1).
\end{align}

Next, we consider the difference $\tccl(T_n,\bgs)-\ccl(T_n,\bgs) $.
A given list $\bP=(p_1,\dots,p_k)\in\hPPl$ belongs to 
$\tCCl(T_n,\bgs)\setminus \CCl(T_n,\bgs)$ if
it satisfies \ref{C1}--\ref{C2} but not \ref{C3}. 
Then the $\sum_1^k(|p_i|+1)=\ell+k$ vertices in the paths belong to at most
$\ell+k - (k+1)=\ell-1$ cycles in $\bgs$. 
It follows from \refL{lem_points_in_cycles},
similarly to \eqref{lc2d} above, that the probability of this event is
\begin{align}
  \label{pilo}
O\bigpar{n^{(\ell-1)-(\ell+k)+1/4)}}=o\bigpar{n^{-k}}.
\end{align}
Hence, arguing as in \eqref{lc2c}--\eqref{pia}, but more crudely, using
\eqref{lc1} and \eqref{pilo}, 
\begin{align}\label{pib}
  \E\bigsqpar{\tccl(T_n,\bgs) -\ccl(T_n,\bgs)} &
=\sum_{|\bm|=\ell} |\PM(T_n)|\frac{1}{2s(\bm)} 
\on{-s(\bm)}
\notag\\&
=o(1),
\end{align}
which proves \eqref{lc2b}.
Together with \eqref{pia}, this also shows \eqref{lc2a},
with 
\begin{align}\label{la}
\gl_\ell&:=
\sum_{|\mathbf{m}|=\ell} 
 \frac{\Pm}{2s(\mathbf{m})}(2\gamma)^{s(\mathbf{m})}.
\end{align}

It remains to calculate this sum, and verify \eqref{gli}.
%Now, it remains to calculate the expectations.
Define the generating functions (\eg{} as formal power series), 
\begin{align}\label{F}
  F(x):=\sum_{m\geq 1} 2mx^m
\end{align}
and
\begin{align}\label{G}
G(x,y):=\sum_{\ell\geq 1} \frac{F(x)^\ell}{2\ell}y^\ell.
\end{align}
Then, by \eqref{eq_Pm}, for any $k\ge1$,
\begin{align}
  \sum_{s(\bm)=k} \Pm x^{|\bm|}=F(x)^k
\end{align}
and thus, using also \eqref{la},
\begin{align}\label{lgl}
%\gL(x):=  
\sum_{\ell\ge1}\gl_\ell x^\ell
=\sum_{\bm}
 \frac{\Pm(2\gamma)^{s(\mathbf{m})}}{2s(\mathbf{m})} x^{|\bm|}
=\sum_{k\geq 1}\frac{(2\gamma)^k}{2k} F(x)^k
=G(x,2\gamma).
\end{align}
On the other hand, we have
\begin{equation}
F(x)=\frac{2x}{(1-x)^2}
\end{equation}
and 
\begin{equation}
G(x,y)=\frac{1}{2}\log\left(\frac{1}{1-F(x)y}\right).
\end{equation}
Hence, \eqref{lgl} yields
\begin{align}\label{lo}
%  \gL(x)=
\sum_{\ell\ge1}\gl_\ell x^\ell=
\frac{1}{2}\log\left(\frac{1}{1-2\gamma\frac{2x}{(1-x)^2}}\right)
=\frac{1}{2}\log\left(\frac{(1-x)^2}{1-(2+4\gamma)x+x^2}\right).
\end{align}
Let $\rho_\pm$ be the roots of $x^2-(2+4\gam)x+1=0$.
Then 
\begin{align}
  \rho_\pm:=1+2\gam\pm\sqrt{4\gam(1+\gam)},
\end{align}
which by \eqref{c1b} yields, after a short calculation,
\begin{align}
  \rho_+=\frac{1+\tau}{1-\tau},
\qquad
  \rho_-=\frac{1-\tau}{1+\tau}.
\end{align}
Furthermore, \eqref{lo} yields
\begin{align}
  \sum_{\ell\ge1}\gl_\ell x^\ell &
= \log(1-x)-\frac12\log\bigpar{(1-\rho_+x)(1-\rho_-x)}
\notag\\&
= \log(1-x)-\frac12\log\bigpar{1-\rho_+x}-\frac12\log\bigpar{1-\rho_-x}
\end{align}
and \eqref{gli} follows,
which concludes the proof.
\end{proof}

We next extend \refL{LC2} and  show convergence also in distribution.

\begin{lemma}
  \label{LC3}
  Let $(T_n)$ be a good sequence of trees.
Then, for every $\ell\ge1$, as \ntoo,
\begin{align}\label{lc3}
 \ccl\tngs \pto\Po(\gl_\ell),
\end{align}
with $\gl_\ell$ given by \eqref{gli}.
Moreover, \eqref{lc3} holds jointly for any finite set of $\ell\ge1$,
with independent limits $\Po(\gl_\ell)$.
\end{lemma}

\textit{Idea of the proof:} This proof is a bit long and technical, but the general idea is quite classical. We will use the method of moments (by considering factorial moments), and roughly speaking, the proof amounts to showing that given a finite number of cycles taken uniformly conditionally on being pairwise distinct, they are pairwise disjoint with high probability. Then, calculating the value of the desired expectations will be straightforward from the results we obtained beforehand. In order to show that our uniform cycles are pairwise disjoint, we will again use their decomposition into lists of paths, and, thanks to Lemmas~\ref{LC1} and~\ref{LC2}, we will only have to consider possible intersections at the endpoints of the paths.

\begin{proof}
First, \eqref{lc2b} implies that $\P\bigpar{\tccl\tngs\neq\ccl\tngs}\to0$;
thus it suffices to show \eqref{lc3} for the simpler $\tccl\tngs$ instead.

We use the method of moments, so we want to estimate the factorial moment
$  \E \prod_{i=1}^q \bigpar{\tcclx{i}(T_n)}_{r_i} $
for any given positive integers $q$, $\ell_1<\dots<\ell_q$ and $r_1,\dots,r_q$.
We argue similarly as in the special case $q=1$ and $r_1=1$  in \refL{LC2},
and write this expectation as
\begin{align}\label{ab}
  \E \prod_{i=1}^q \bigpar{\tcclx{i}(T_n)}_{r_i} 
=\hsumx_{(\bP(i,j))_{ij}} \pi((\bP(i,j))_{ij}),
\end{align}
where we sum over all sequences of distinct lists 
%$(\bP(i,j))_{\substack{1\le i\le  q\\1\le j\le r_i}}$ 
$(\bP(i,j))_{1\le i\le  q,\ 1\le j\le r_i}$ 
such that 
$\bP(i,j)\in \hPPlx{i}(T_n)$, 
and $\pi((\bP(i,j))_{ij})$ is the probability that every
$\bP(i,j)\in\tCClx{i}(T_n)$. 
Recalling the definition of $\hPPl$, 
we can rewrite \eqref{ab} as
\begin{align}\label{ac}
  \E \prod_{i=1}^q \bigpar{\tCClx{i}(T_n)}_{r_i} 
= \sumx_{(\bP(i,j))_{ij}}\frac{ \pi((\bP(i,j))_{ij})}{\prod_{i,j}{2s(\bP(i,j))}},
\end{align}
where we now sum over all sequences of  lists 
$(\bP(i,j))_{ij}$ such that 
$\bP(i,j)\in \PPlx{i}(T_n)$ and no two $\bP(i,j)$ are equivalent (for
$\simeqx$).

First, we show that the $\Ext(\bP(i,j))$ are pairwise disjoint whp. 

For each such sequence $(\bP(i,j))_{ij}$, define a graph $H$ with vertex
set $V(H):=\bigcup_{i,j}\Ext(\bP(i,j))$, the set of endpoints of all
participating paths, and edges of two colours as follows:
For each list $\bP(i,j)=(p_\mnu)_1^k$, add for each $\mnu$ a \emph{green} edge
between $\start(p_\mnu)$ and $\eend(p_\mnu)$, 
and a \emph{red} edge between $\eend(p_\mnu)$ and $\start(p_{\mnu+1})$ (see Figure~\ref{fig_path_graph} for an example). 
(We use here and below the convention $p_{k+1}:=p_1$.) 
Hence, by the definitions above,
every $\bP(i,j)\in\tCClx{i}(T_n)$
if and only if 
each red edge in $H$ joins two vertices in the same cycle of $\bgs$.
Thus, $\pi((\bP(i,j))_{ij})$ in \eqref{ac} is the probability of this event.

\begin{figure}
\center
\includegraphics[scale=0.8]{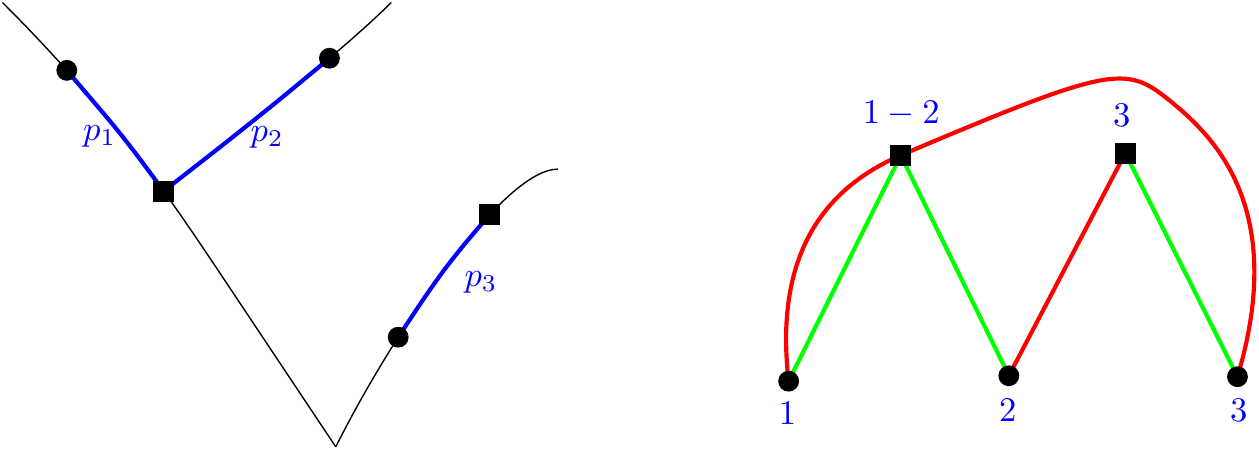}
\caption{Three paths in a tree (left) and their associated graph $H$ (right). The paths are in blue, the start of a path is represented as a square and its end, as a dot.}\label{fig_path_graph}
\end{figure}

For each  graph $H$ constructed in this way, let $H_G$ %[$H_R$]
be the subgraph consisting of all green %[red] 
edges, and say that a connected component of $H_G$ is a \emph{green
  component} of $H$. 
Define a \emph{red component} in the same way, and let $\gq_G(H)$ and $\gq_R(H)$ be
the numbers of green and red components, respectively.

Let $M_H=M_H(n)$ be the number of terms in \eqref{ac} with a given graph $H$.
(For some fixed $q$, $\ell_1,\dots,\ell_q$ and $r_1,\dots,r_q$.)
We estimate $M_H$ as follows.
Each green component of $H$ corresponds to some paths $p_{i,j,\mnu}$ such that
their union is a connected subtree $\bt$ of $T_n$. All these paths have
lengths bounded by $\max_i\ell_i$, so the union $\bt$ has bounded size.
Thus, regarding $\bt$ as an unlabelled tree $\bbt$, 
there is only a finite set of possible choices of $\bbt$. 
Hence, the assumption \eqref{ant} implies that there is $O(n)$ possible 
choices of the subgraph $\bt$ in $T_n$, for each green component in $H$. 
Moreover, for each choice of $\bt$,
there are $O(1)$ choices of each path $p_{i,j,\mnu}$ in $\bt$, so for each green
component in $H$, the paths $p_{i,j,\mnu}$ that correspond to edges in the
component can be chosen in $O(n)$ ways.
Consequently, we have
\begin{align}\label{mh}
  M_H= O\bigpar{n^{\gq_G(H)}}.
\end{align}

Moreover, we have seen that $\pi((\bP(i,j))_{ij})$ in \eqref{ac} is the
probability that each red component lies in a single cycle of $\sig$.
This entails that the $v(H)$ vertices in $H$
lie in at most $\gq_R$ different cycles of $\bgs$,
and thus \refL{lem_points_in_cycles} shows that
\begin{align}\label{ad}
  \pi((\bP(i,j))_{ij}) = O\bigpar{n^{\gq_R(H)-v(H)+1/4}}.
\end{align}

Consequently, the total contribution to \eqref{ac} for all sequences of
lists yielding a given $H$ is, by \eqref{mh} and \eqref{ad},
\begin{align}
  \label{ae}
O\bigpar{n^{\gq_G(H)+\gq_R(H)-v(H)+1/4}}.
\end{align}
Since each green or red component has size at least 2, it follows that
$v(H)\ge 2\gq_G(H)$ and $v(H)\ge 2\gq_R(H)$, and thus $\gq_G(H)+\gq_R(H) \le v(H)$.
If we here have strict inequality, then \eqref{ae} shows that the
contribution is $o(1)$ and may be ignored. (There is only a finite number of
possible $H$ to consider.)

Hence, it suffices to consider the case $\gq_G(H)=\gq_R(H)=v(H)/2$. 
This implies that all green or red components have size 2, and thus are
isolated edges. It follows that if two 
different lists $\bP(i_1,j_1)$ and $\bP(i_2,j_2)$ contain two paths
$p_{i_1,j_1,\mnu_1}$ and $p_{i_2,j_2,\mnu_2}$ 
that have a common endpoint, then these paths have to coincide
(up to orientation).
Furthermore, if they coincide, and have, say, the same orientation so
$\eend(p_{i_1,j_1,\mnu_1})=\eend(p_{i_2,j_2,\mnu_2})$, then the red edges
from that vertex have to coincide, so 
$\start(p_{i_1,j_1,\mnu_1+1})=\start(p_{i_2,j_2,\mnu_2+1})$.
It follows easily 
that the two lists $\bP(i_1,j_1)$ and $\bP(i_2,j_2)$ are equivalent
in the sense $\bP(i_1,j_1)\equiv\bP(i_2,j_2)$ defined above.
However, we have excluded this possibility, and this contradiction shows
that  all paths $p_{i,j,\mnu}$ in the lists have disjoint sets of endpoints 
$\Ext(p_{i,j,\mnu})$. 

Consider now this case, and assume, moreover, that
for all $i$ and $j$, we have $s(\bP(i,j))=k(i,j)$
and $\bP(i,j)\in \PMx{\bm(i,j)}$ for some given $k(i,j)$ and
$\bm(i,j)=(m_{i,j,1},\dots,m_{i,j,k(i,j)})$ such that $|\bm(i,j)|=\ell_i$.
Then we can write $\bP(i,j)=(p_{i,j,1},\dots,p_{i,j,k(i,j)})$, where all paths
$p_{i,j,\mnu}$ have disjoint endpoints and $|p_{i,j,\mnu}|=m_{i,j,\mnu}$.
It follows as in the proof of \refL{LC1} that the number of such 
sequences of lists $(\bP(i,j))_{ij}$ is, with $K:=\sum_{i,j}k(i,j)$, the
total number of paths in the lists,
and using \eqref{eq_Pm},
\begin{align}\label{af}
  \prod_{i,j,\mnu} (2m_{i,j,\mnu}n) + \on{K}
=  \prod_{i,j} \Pmx{i,j} n^K + \on{K}
\end{align}
Furthermore, in this case,
the set $V(H)$ is the disjoint union of all
$\Ext(p_{i,j,\nu})$, which have 2 elements each, and thus
\begin{align}\label{ag}
  v(H)=2\sum_{i,j}k(i,j)=2K.
\end{align}
Hence, $\pi\bigpar{(\bP(i,j))_{ij}}$ is the probability 
that $K$ given pairs of vertices in $T_n$ belong to the same cycles in $\bgs$.
As in the proof of \refL{LC2}, it follows from \refLs{lem_proba_merging} and
\ref{lem_points_in_cycles} that, using \eqref{la},
\begin{align}\label{ah}
\pi\bigpar{(\bP(i,j))_{ij}}
  = (2\gam)^K n^{-K} + \on{-K}.
\end{align}

From \eqref{ac}, \eqref{af} and \eqref{ah}, we now obtain
\begin{align}\label{ach}
  \E \prod_{i=1}^q \bigpar{\tCClx{i}(T_n)}_{r_i} 
%\\
&=
\sum_{|\bm(i,j)|=\ell_i \forall i,j}
\frac{\prod_{i,j}\Pmx{i,j}}{\prod_{i,j}{2s(\bm(i,j))}}(2\gam)^K + o(1)
\notag\\
&=
\sum_{|\bm(i,j)|=\ell_i \forall i,j}
\prod_{i,j}\frac{\Pmx{i,j}(2\gam)^{s(\bm(i,j))}}{2s(\bm(i,j))}+ o(1)
\notag\\
&=
\prod_{i,j}\gl_{\ell_i}+o(1)
=\prodik \gl_{\ell_i}^{r_i}+o(1).
\end{align}

This proves the desired convergence of the factorial moments, and
the method of moments yields the result.
(For $\tccl\tngs$, which is enough as said at the beginning of the proof.)
\end{proof}

To prove the main theorem, %\refT{thm_cycles}. 
it remains only to replace the deterministic trees
$T_n$ in \refL{LC3} by the random $\T$.

\begin{proof}[Proof of Theorem~\ref{thm_cycles}]

Recall that, for each $n$, $\T=\T_n$ is a uniformly random tree in $\cT_n$.
By \refLs{lem_paths} and \ref{LNt}, 
\eqref{apl} and \eqref{ant} hold in probability if we take $T_n$ as the
random tree $\T$.
We can regard this as convergence in probability of infinite sequences
indexed by the countable set $\set{\ell}\cup\set{\bt}$, \ie, convergence in
$\bbR^\infty$. 
By the Skorohod coupling theorem \cite[Theorem~4.30]{Kallenberg},
we may assume that this convergence actually holds almost surely, i.e., that
\eqref{apl} and \eqref{ant} hold \as{} for every $\ell$ and $\bt$.
In other words, we may assume that the random trees $\T_n$ for different $n$
are coupled such that the sequence $(\T_n)$ is good a.s.
%(An alternative argument is that \eqref{apl} and \eqref{ant} imply
%convergence \as{} for some subsequence, and by a diagonal argument we  may
%find a subsequence where \eqref{apl} and \eqref{ant} hold \as{} for all
%$\ell$ and $\bt$, i.e., $\T_n$ is good along the subsequence. 
%The result then follows as below along the subsequence,
%and the full result follows by a standard subsubsequence argument, see \eg{}
%\cite[Section 5.7]{Gut}.)

We now condition on the sequence $(\T_n)$. We have just seen that \refL{LC3}
applies a.s., and thus 
\eqref{tc} holds for the conditional distributions, i.e.
\begin{align}\label{tcc}
\bigpar{(C_{n,g_n}^{(1)},C_{n,g_n}^{(2)},\ldots,C_{n,g_n}^{(M)}) \mid \T_n}
\dto \Po(\lambda_1,\lambda_2,\ldots,\lambda_M)
\qquad \text{a.s.}
\end{align}
This immediately implies the unconditional \eqref{tc}.
\end{proof}

\section{Vertex degrees}\label{Sdegree}

In this section, we prove Theorem~\ref{TD}.
Again, we may by \refT{thm_ctrees} consider the underlying graph of $(\T,\sig)$.
It is convenient to first consider deterministic trees and permutations,
with only the relative position of them random.

\begin{lemma}\label{LD}
  Let $(g_n)$ be a sequence with $n-2g_n\to\infty$ as \ntoo.
  Let $(T_n,\gs_n)$ be a given (deterministic)
sequence, where $T_n$ is a plane tree of size
  $n$, and $\gs_n$ is a permutation of $[n+1]$ with exactly $m$ cycles, 
where $m=m_n:=n+1-2g_n$.
Let $\T_n$ be $T_n$ with the vertices labeled $1,\dots,n+1$ uniformly at
random, and let 
%$G_n$ be the underlying graph of the
$\NNng{k}$ be the number of vertices of degree $k$ in 
the underlying graph of the
corresponding decorated tree $(\T_n,\gs_n)$.

Let $n\kkk_n$ be the number of vertices of degree $k$ in $T_n$, and let
$m\kkl_n$ be the number of cycles of length $\ell$ in $\gs_n$.
Suppose that $(\rho_k)$ and $(\pi_\ell)$ are probability distributions on
\set{1,2,\dots} such that, as \ntoo, 
\begin{align}
  \label{ld1a}
n\kkk_n/n&\to \rho_k, \quad k\ge1,
\\\label{ld1b}
m\kkl_n/m&\to \pi_\ell, \quad \ell\ge1.
\end{align}
Then,
\begin{align}\label{ld2}
  \NNng{k}/m\pto \kk_k,
\qquad k\ge1,
\end{align}
where
$(\kk_k)$ is a probability distribution 
which equals the distribution of the random sum $\sum_{i=1}^M Y_i$, where
$M, Y_1, Y_2,\dots$ are independent, $M$ has the distribution $(\pi_\ell)$
and each $Y_i$ has the distribution $(\rho_k)$.
Equivalently, $(\kk_k)$ is 
given by 
the probability generating function
\begin{align}\label{ld3}
H_\kk(x):=\sumk \kk_k x^k = H_\pi\bigpar{H_\rho(x)},
\end{align}
where $H_\rho$ and $H_\pi$ are the \pgf{s} of $(\rho_k)$ and $(\pi_\ell)$.
\end{lemma}

\begin{proof}
Let $\bGn$ be the underlying graph of $(\T_n,\gs_n)$.
Fix $\ell$ and number the %$m\kkl_n$ 
cycles in $\gs_n$ of length $\ell$ as $\cC_1,\dots,\cC_{\mnl}$.
For each $i\le m\kkl_n$, let $X_{i1}, \dots,X_{i\ell}$ be the degrees in
$\T_n$ of the vertices covered by $\cC_i$, taken in uniformly random order,
and let $D_i=\sum_{j=1}^\ell X_{ij}$ be the degree of the corresponding
vertex in $\bGn$.

For any sequence $\bd=(d_1,\dots,d_\ell)\in\bbZ_+^\ell$ and $i\le\mnl$, let
\begin{align}\label{pc}
I_{\bd,i}:=\indic{(X_{i1},\dots,X_{i\ell})=\bd}.  
\end{align}
Also, let, as in the statement,  
$Y_1,Y_2,\dots$ be independent random variables with the distribution
$(\rho_k)$, and define
\begin{align}\label{pd}
  p_\bd:=\P\bigpar{(Y_{1},\dots,Y_{\ell})=\bd}=\prod_{j=1}^\ell \rho_{d_j}.
\end{align}
Note that the vertex degrees $X_{11},\dots,X_{1\ell},X_{21},\dots,X_{2\ell},\dots$
are the degrees of a sequence of vertices of $T_n$ that are picked at random
without replacement. Hence we have, by symmetry and \eqref{ld1a},
\begin{align}\label{pe}
  \E I_{\bd,i}
=\E I_{\bd,1}
&=\frac{n_n\kkx{d_1}(n_n\kkx{d_2}+O(1))\dotsm(n_n\kkx{d_\ell}+O(1))}
{(n+1)n\dotsm (n-\ell+2)}
\notag\\&
\to \prod_{j=1}^\ell \rho_{d_j} =p_\bd.
\end{align}
Similarly, for any $i\neq j$,
\begin{align}\label{pf}
  \E \bigsqpar{I_{\bd,i}I_{\bd,j}}
=\E \bigsqpar{I_{\bd,1}I_{\bd,2}}
\to p_\bd^2,
\end{align}
and thus
\begin{align}\label{pfc}
  \Cov\bigpar{I_{\bd,i},I_{\bd,j}}
=  \Cov\bigpar{I_{\bd,1},I_{\bd,2}}
%=  \E \bigsqpar{I_{\bd,i}I_{\bd,j}}
%-\E \bigsqpar{I_{\bd,1}}\E\bigsqpar{I_{\bd,2}}
\to 0
.\end{align}
Hence,
if $\sd:=\sum_{i=1}^{\mnl} I_{\bd,i}$, then, using also \eqref{ld1b},
\begin{align}\label{pg}
\frac{\E \sd}{m} = \frac{\mnl}{m} \E I_{\bd,1} \to \pi_\ell p_\bd
\end{align}
and
\begin{align}\label{ph}
  \Var \sd
= \mnl \Var I_{\bd,1}+ \mnl(\mnl-1)\Cov\bigpar{I_{\bd,1},I_{\bd,2}}
=o(m^2).
\end{align}
Consequently, by Chebyshev's inequality and \mtoo,
\begin{align}
  \label{au}
\sd/m\pto \pi_\ell p_\bd.
\end{align}

Let $|\bd|:=d_1+\dots+d_\ell$. Then, by our definitions,
\begin{align}
  \NNng{k}=\sumlk \sum_{\substack{\bd\in\bbZ_+^\ell\\|\bd|=k}}\sd.
\end{align}
Note that, for any given $k$, this is a finite sum. Hence, \eqref{au}
implies, recalling \eqref{pd},
\begin{align}\label{pt}
  \NNng{k}/m \pto \sum_\ell
  \sum_{\substack{\bd\in\bbZ_+^\ell\\|\bd|=k}}\pi_\ell p_\bd
= \sum_\ell   \sum_{\substack{\bd\in\bbZ_+^\ell\\|\bd|=k}}\pi_\ell p_\bd
=\sum_\ell\pi_\ell  \P\Bigpar{\sum_{i=1}^\ell Y_i = k}.
\end{align}
This proves \eqref{ld2}, with the limit $\kk_k$ given by the \rhs{} of
\eqref{pt}. This clearly equals $\P(\sum_1^M Y_i=k)$, with $M$ as in the
statement. 
Furthermore, \eqref{pt} implies that the generating function $H_\kk$ is
given by
\begin{align}
  H_\kk(x)
&=\sumk x^k\suml\pi_l \P\Bigpar{\sum_{i=1}^\ell Y_i = k}
=\suml\pi_l \E x^{\sum_{i=1}^\ell Y_i}
\notag\\&
=\suml\pi_l \bigpar{\E x^{Y_1}}^\ell
=\suml\pi_l H_\rho(x)^\ell
=H_\pi(H_\rho(x)).
\end{align}
\end{proof}

\begin{proof}  [Proof of \refT{TD}]
We apply \refT{thm_ctrees} and consider the underlying graph of $(\T,\sig)$.
Furthermore, we may do an extra randomization as in \refL{LD} of the way the
tree $\T$ is decorated by the permutation $\sig$; this obviously will not
change the distribution of $\NNng{k}$.
Now condition on $\T$ and $\sig$, and note that the conditional distribution
of $\NNng{k}$ is as in \refL{LD}, with $T_n=\T$ and $\gs_n=\sig$.
These are random, and thus $n\kkk_n$ and $m\kkl_n$ are now random.

Proposition~\ref{T1} shows that \eqref{ld1b} holds in probability:
\begin{align}\label{qa}
  m\kkl_n/m&\pto \pi_\ell,
\qquad\ell\ge1.
\end{align}
with $(\pi_\ell)$ given by \eqref{pi}, 
where $\tau$ is given by \eqref{tau} with
$\ga=1-2\gth$.

Furthermore, it is well-known that the degree distribution in a random plane
tree is asymptotically $\Ge(1/2)$, see \eg{} 
\cite[Section 3.2.1]{Drmota} or
\cite[Theorem 7.11(ii) and Example 10.1]{SJ264}; more precisely,
\eqref{ld1a} holds in probability:
\begin{align}\label{qb}
  n\kkk_n/n \pto \rho_k:= 2^{-k}, \qquad k\ge1.
\end{align}

By the Skorokhod coupling theorem \cite[Theorem~4.30]{Kallenberg},
we may assume that \eqref{qa} and \eqref{qb} hold a.s., for all $k$ and
$\ell$, and then 
\refL{LD} apples after conditioning on $\T$ and $\sig$.
Consequently, \eqref{ld2} holds conditioned on $\T$ and $\sig$, a.s., which
implies that \eqref{ld2} holds without conditioning.
This proves \eqref{td}, and it remains only to identify the limit
distribution.

We have, by \eqref{qb},
\begin{align}
  H_\rho(x)=\sumk 2^{-k}x^k=\frac{x}{2-x},
\end{align}
and  by \eqref{pi} and \eqref{Phi},
\begin{align}
  H_\pi(x)
=\frac{\Phi(\tau x)}{\Phi(\tau)}
=\frac{1}{2\Phi(\tau)}\bigpar{\log(1+\tau x)-\log(1-\tau x)}
.\end{align}
Hence, \eqref{ld3} yields
\begin{align}
  H_\kk(x)&
=H_\pi\Bigpar{\frac{x}{2-x}}
=\frac{1}{2\Phi(\tau)}
\Bigpar{\log\Bigpar{1+\tau \frac{x}{2-x}}-\log\Bigpar{1-\tau \frac{x}{2-x}}}
\notag\\&
=\frac{1}{2\Phi(\tau)}
\Bigpar{\log\Bigpar{2-x+\tau x}-\log\Bigpar{2-x-\tau x}}
\notag\\&
=\frac{1}{2\Phi(\tau)}
\Bigpar{\log\Bigpar{1-\frac{1-\tau}2 x}-\log\Bigpar{1-\frac{1+\tau}2x}}.
\end{align}
This yields \eqref{td2}, recalling again \eqref{Phi}.
\end{proof}

\bibliographystyle{abbrv}
%\bibliography{bibli}

\begin{thebibliography}{10}

\bibitem{ACCR13}
O.~Angel, G.~Chapuy, N.~Curien, and G.~Ray.
\newblock The local limit of unicellular maps in high genus.
\newblock {\em Electron. Commun. Probab.}, 18(86):1--8, 2013.

\bibitem{AS03}
O.~Angel and O.~Schramm.
\newblock Uniform infinite planar triangulations.
\newblock {\em Comm. Math. Phys.}, 241(2-3):191--213, 2003.

\bibitem{BC86}
E.~A. Bender and E.~Canfield.
\newblock The asymptotic number of rooted maps on a surface.
\newblock {\em Journal of Combinatorial Theory, Series A}, 43(2):244 -- 257,
  1986.
  
\bibitem{Bet16}
J.~Bettinelli.
\newblock Geodesics in Brownian surfaces (Brownian maps).
\newblock {\em Ann. Inst. Henri Poincar\'e Probab. Stat.}, 52(2):612-646, 2016.


\bibitem{BL19}
T.~Budzinski and B.~Louf.
\newblock Local limits of uniform triangulations in high genus.
\newblock {\em Invent. Math.}, 223(1):1--47, 2021.

\bibitem{BL20}
%\texttt{BL20}
T.~Budzinski and B.~Louf.
\newblock Planarity and non-separating cycles in uniform high genus quadrangulations.
\newblock Preprint, 2020.
\texttt{arXiv}:2012.05813.

\bibitem{CFF13}
G.~Chapuy, V.~F\'eray, and E.~Fusy.
\newblock A simple model of trees for unicellular maps.
\newblock {\em Journal of Combinatorial Theory, Series A}, 
120:2064--2092, 2013.

\bibitem{CMS09}
G.~Chapuy, M.~Marcus, and G.~Schaeffer.
\newblock A bijection for rooted maps on orientable surfaces.
\newblock {\em SIAM J. Discrete Math.}, 23(3):1587--1611, 2009.

\bibitem{CVHM15}
E.~Colin~de Verdi\`ere, A.~Hubard, and A.~de~Mesmay.
\newblock Discrete systolic inequalities and decompositions of triangulated
  surfaces.
\newblock {\em Discrete Comput. Geom.}, 53(3):587--620, 2015.

\bibitem{Drmota}
M. Drmota.
\newblock \emph{Random trees}.
\newblock  SpringerWienNewYork, Vienna, 2009.


\bibitem{Gro92}
M.~Gromov.
\newblock Systoles and intersystolic inequalities.
\newblock In {\em Actes de la {T}able {R}onde de {G}\'{e}om\'{e}trie
  {D}iff\'{e}rentielle ({L}uminy, 1992)}, volume~1 of {\em S\'{e}min. Congr.},
  pages 291--362. Soc. Math. France, Paris, 1996.

\bibitem{Gut}
A.~Gut.
\newblock {\em Probability: A Graduate Course}.
\newblock Springer New York, 2013.

\bibitem{SJ264}
S.~Janson.
\newblock Simply generated trees, conditioned galton–watson trees, random
  allocations and condensation.
\newblock {\em Probab. Surveys}, 9:103--252, 2012.

\bibitem{SJ355}
S. Janson.
\newblock On general subtrees of a conditioned Galton--Watson tree.
\newblock Preprint, 2020.
\texttt{arXiv}:2011.04224.
%	arXiv:2011.04224
%\bibitem{}


\bibitem{Kallenberg}
O. Kallenberg.
\newblock \emph{Foundations of Modern Probability}.
2nd ed., Springer, New York, 2002. 


\bibitem{LG11}
J.-F. Le~Gall.
\newblock Uniqueness and universality of the {B}rownian map.
\newblock {\em Ann. Probab.}, 41:2880--2960, 2013.

\bibitem{Lou20}
%\texttt{Lou20} 
B.~Louf.
\newblock Planarity and non-separating cycles in uniform high genus quadrangulations.
\newblock Preprint, 2020.
\texttt{arXiv}:2012.06512.

\bibitem{Lou21}
%\texttt{Lou20} 
B.~Louf.
\newblock Large expanders in high genus unicellular maps.
\newblock Preprint, 2021.
\texttt{arXiv}:2102.11680.

\bibitem{Mie11}
G.~Miermont.
\newblock The {B}rownian map is the scaling limit of uniform random plane
  quadrangulations.
\newblock {\em Acta Math.}, 210(2):319--401, 2013.

\bibitem{Mir13}
M.~Mirzakhani.
\newblock Growth of {W}eil--{P}etersson volumes and random hyperbolic surfaces
  of large genus.
\newblock {\em J. Differential Geom.}, 94(2):267--300, 2013.

\bibitem{MP19}
M.~Mirzakhani and B.~Petri.
\newblock Lengths of closed geodesics on random surfaces of large genus.
\newblock {\em Comment. Math. Helv.}, 94(4):869--889, 2019.

\bibitem{Ray13a}
G.~Ray.
\newblock Large unicellular maps in high genus.
\newblock {\em Ann. Inst. H. Poincaré Probab. Statist.}, 51(4):1432--1456, 11
  2015.

\bibitem{Sch98these}
G.~Schaeffer.
\newblock {\em Conjugaison d'arbres et cartes combinatoires al\'eatoires}.
\newblock Th\`ese de doctorat, Universit\'e Bordeaux I, 1998.

\bibitem{Tut63}
W.~T. Tutte.
\newblock A census of planar maps.
\newblock {\em Canad. J. Math.}, 15:249--271, 1963.






\end{thebibliography}

\end{document}